\documentclass[reqno]{amsart}
\usepackage{amssymb}
\usepackage{amsmath}
\usepackage{mathtools}
\usepackage{fixltx2e}
\usepackage{bbm}
\usepackage[cyr]{aeguill}
\usepackage{cmap}
\usepackage{hyperxmp}
\newtheorem{thm}{Theorem}[section]

\newtheorem{prop}[thm]{Proposition}

\theoremstyle{definition}
\newtheorem{defn}[thm]{Definition}
\theoremstyle{remark}
\newtheorem{rem}[thm]{Remark}

\numberwithin{equation}{section}
\newcommand{\RR}{\mathbb{R}}

\newcommand{\Circle}{\mathbb{S}^{1}}

\newcommand{\DiffS}{\mathrm{Diff}_{+}^{\infty}(\mathbb{S}^{1})}

\newcommand{\DM}[1]{\mathcal{D}^{#1}(M)}
\newcommand{\D}[1]{\mathcal{D}^{#1}(\mathbb{S}^{1})}

\newcommand{\DiffM}{\mathrm{Diff}^{\infty}(M)}

\newcommand{\CS}{\mathrm{C}^{\infty}(\mathbb{S}^{1})}

\newcommand{\cinf}{c^{\infty}}

 \newcommand{\widebar}[1]{\mkern 1.5mu\overline{\mkern-1.5mu#1\mkern-1.5mu}\mkern 1.5mu}

\newcommand{\set}[1]{\left\{#1\right\}}

\newcommand{\f}{\varphi} %
\newcommand{\p}{\psi}

\DeclareMathOperator{\ad}{ad}
\DeclareMathOperator{\Ad}{Ad} %
\newcommand{\E}{\mathbb{E}}
\newcommand{\F}{\mathbb{F}}

\newcommand{\Cinf}{\mathrm{C}^{\infty}}

\title[A Spray Theory for the Geometric Method ] 
      {A Spray Theory for the Geometric Method   in Hydrodynamics}

\author[Emanuel-Ciprian Cismas ]{}


 \email{emanuel.cismas@upt.ro}

\thanks{The first author is supported by NSF grant xx-xxxx}

\thanks{$^*$ Corresponding author: xxxx}

\begin{document}
\maketitle

\centerline{\scshape Emanuel-Ciprian Cismas $^*$}
\medskip
{\footnotesize
 \centerline{Universitatea "Politehnica", }
 \centerline{Timisoara, Piata Victoriei 2,  300006, Romania}
} 

\medskip

\begin{abstract} We present a brief spray theory necessary  for the geometric method in hydrodynamics and shape analysis. We make use of the convenient calculus to complete a unitary approach started  by P. Michor and A. Kriegl.

\end{abstract}

\section{Introduction}
\label{sec:introduction}

\space The main goal of this article is to fix the lack o rigorousness in  a few papers involving infinite dimensional Lie groups. For infinite dimensional Lie groups we have to do the things the right way:  the  "convenient way".

 We will study infinite dimensional manifolds modelled on locally convex spaces. As is well-known, beyond Banach spaces a lot of pathologies occur: ordinary differential equations may not have solutions or  the solutions may not be unique, there is no genuine inverse function theorem,  there is no natural topology for the dual space and none of the candidates is metrizable.  All these problems oblige us to handle carefully   the geometric objects related to an infinite dimensional Fr\'{e}chet manifold. There was a common belief between mathematicians that for infinite dimensional calculus each serious application
needs its own foundation until 1982 when A. Fr\"olicher and A. Kriegl presented independently  the solution to the question regarding the right differential calculus in infinite dimension: \textit{the convenient calculus}. P. Michor and A. Kriegl laid afterwards the foundations of the infinite dimensional differential geometry and brought everything together in their seminal book \cite{Michor}.

Twenty years before, in his influential article \cite{Arn1966}, V. Arnold had the idea to analyze the motion of hydrodynamical systems using geodesic flows. Actually he showed that the Euler equations of hydrodynamics can be recast as geodesic equations of a right-invariant Riemannian metric on the group of volume-preserving diffeomorphisms. Nowadays, similar methods are used in shape analysis \cite{Younes}. This approach became the so called \textit{geometric method in hydrodynamics} (see \cite{Escher-Kolev} for more details) and involves the use of geometric arguments to study issues like well-posedness or stability.

For example,  the Camassa-Holm equation

$$u_t-u_{txx}=2u_xu_{xx}+uu_{xxx}-3uu_{x}\hspace{1cm}$$
describes the geodesic flow on the group $ \DiffS$ of  orientation-preserving diffeomorphisms of the circle. The right-invariant metric considered is obtained by right-translations from the inner product
$$\langle u,v\rangle=\int_{\Circle}uv+u_xv_x\ dx,\ \ \ u,v\in T_{id} \DiffS,$$
One can observe that the differential operator $A=I-D^2$, called \textit{inertia operator} in this context, generates the inner product by
$$\langle u,v\rangle=\int_{\Circle}u\cdot Av\ dx.$$

There is a high flexibility of this method given by the choice of  an inertia operator or by the choice of an algebraic structure, usually involving diffeomorphism groups. In the last decade different inertia operators were studied starting with differential operators with constant coefficients \cite{CK2003}, Hilbert transforms \cite{EKW2012,Wun2010} or Fourier multipliers \cite{BauerEscherKolev,EKHS} and even pseudo-differential operators \cite{Cismas2016}. Among the algebraic structures studied one should mention homogeneous spaces \cite{Lenels}, the Bott-Virasoro group \cite{MichorRatiu},  semi-direct products between a group and a vector space \cite{Kohlmann,Vizman} or semidirect products between two groups \cite{Cismas2014,Cismas2016}.

The idea behind this geometric approach, initially developed by D.Ebin and J. Marsden in \cite{EM1970}, is to use the right-invariance of the geodesic spray to obtain, via a "no gain, no loss" result (see \cite{Escher-Kolev} for details) a Cauchy-Lipschitz type theorem on a Fr\'{e}chet space.  Using this method we avoid the Nash-Moser schemes in order to obtain well-posedness in the smooth category.  To formulate an inverse function theorem for Fr\'{e}chet spaces  we have to restrict to the category of tame Fr\'{e}chet spaces,  in the sense of R.S. Hamilton \cite{Ham1982}.  For example,  the results obtained for $H^{\infty}(\RR^d,\RR^d)$ in \cite{BauerEscherKolev}, with geometric arguments, can not be obtained with a classical Nash-Moser approach, since $H^{\infty}(\RR^d,\RR^d)$  is no longer a tame space.

The article of V. Arnold is not rigorous enough for the infinite dimensional case. Most of the research work influenced by \cite{Arn1966} uses the concept of G\^ateaux smoothness. This concept of smoothness can be satisfactory for some problems but it is not the appropriate one to use in manipulating vector bundles or differential forms. The convenient smoothness defined below fixes this problem. In \cite{MichorRatiu} and \cite{Michor} the authors started to present a solid foundation of the aforementioned geometric method in hydrodynamics.  There is still a missing piece: an introduction via convenient calculus of the geodesic spray theory. Thus, in Section \ref{Sec:6}, we will add the last piece of the puzzle.

\section{Notations} In this article we will use the following notations:
\begin{itemize}
	\item the letters  $\E,$ $\F,$ $\mathbb{G}$ will be used to denote vector spaces.
	\item (spaces of linear mappings)  $\mathrm{L}(\mathbb{E},\mathbb{F})$ denotes the space of linear and bounded mappings that in general do not coincide with the space $\mathcal{L}(\mathbb{E},\mathbb{F})$ of linear and continuous mappings, since $\E$ and $\F$ might not be bornological spaces.
\item (dual spaces) because sometimes we work with locally convex spaces which may not be bornological  we have two notions for the dual of a locally convex space $\E:$ the bornological dual, denoted $\E^{'},$ i.e. the set of all bounded linear functionals $f:\mathbb{E}\rightarrow \RR,$ and the topological dual, denoted $\E^*,$  which is the set of all continuous linear functionals. 

	\item (multiple derivatives) since we discuss different differentiability concepts we denote with $Df$ the Fr\'{e}chet derivative, with $\partial f$ the G\^ateaux derivative, and with $df$ the derivative in the convenient setting.
	
		\item $G$ is a Lie group, $\mathfrak g$ its Lie algebra and $\kappa^r$ is the right Maurer-Cartan form.
	
	\item (different objects) for elements in $\mathfrak g$, or in some modelling space $\mathbb E$, we use small letters  like $u,v$ or $x,y$, for vector fields greek letters $\xi, \eta$, and for second order vector fields capital letters like $X,Y.$
	
		\item (too many $r$'s) $\mathcal R$ is the bornological isomorphism between the convenient vector spaces $C^\infty(G,\mathfrak g)$ and $\Gamma(TG)$ and $R_g$ are the right translations on $TG,$ the tangent mappings of $r_g(h):=hg,$ for $\ g,h\in G.$

	\item if $A,B,C$ are sets for two mappings $f:A\rightarrow C^B$ and $g:A\times B\rightarrow C$ one can define the cannonically attached mappings, sometimes called adjoint mappings:
	$$f^{\wedge}: A\times B\rightarrow C, \hspace{0,3cm}f^{\wedge}(a,b):=f(a)(b), \hspace{0,2cm}a\in A, b\in B,$$
	$$g^{\vee}: A\rightarrow C^B, \hspace{0,3cm}g^{\vee}(a):=g(a,\cdot), \hspace{0,2cm}a\in A.$$

\end{itemize}

\section{Smooth differentiable mappings}
To discuss about a smooth structure of an infinite dimensional topological manifold we need a notion of differentiability between Fr\'{e}chet spaces. In  Banach spaces we have a notion of differentiability, called \textit{Fr\'{e}chet differentiability}, which permits us to extend  the differential calculus from  finite  dimension to infinite dimensional Banach spaces.
\begin{defn} Let $\E,\F$ be Banach spaces, and $U$ an open subset of $\E$. A mapping $f$ is said to be \textit{differentiable}  at a point $x\in U$, if there is an element $A_x\in\mathcal{L}(\E,\F)$ such that
	$$\lim_{h\rightarrow 0}\frac{\|f(x+h)-f(x)-A_x(h)\|}{\|h\|}=0.$$
\end{defn}
In this way  $f$ can be approximated locally by an affine mapping generated by the linear mapping $A_x,$ usually denoted by $D_xf.$ If $f$ is differentiable at every point $x\in U$, then $Df$ can be regarded as a mapping of $U$ into $\mathcal{L}(\E,\F)$, and $f$ is a $C^1$ differentiable mapping if and only if $Df$ is continuous. Since $\mathcal{L}(\E,\F)$ is a Banach space, a $C^k$ differentiable mapping can be defined inductively.

When $\E, \F$ are non-normable Fr\'{e}chet spaces  we have to cope with the following phenomenon, namely the composition
$$ \circ : \mathcal{L}(\mathbb{F},\mathbb{G})\times \mathcal{L}(\mathbb{E},\mathbb{F})\rightarrow \mathcal{L}(\mathbb{E},\mathbb{G}) $$ is not continuous for any locally convex topology which can endow the space of linear mappings, excepting the case when all the spaces are Banach. If we define a concept of smoothness that  uses  the continuity of the mapping
 $$Df: U\rightarrow \mathcal{L}(\E,\F)$$
  the concept will not be conserved by compositions, i.e. there will be no chain rule. For a discussion on this topic one can consult \cite{Bastiani} or \cite{Michor}.

The most used concept of smoothness for infinite dimensional manifolds, modelled on locally convex topological vector spaces, see \cite{Ham1982,Mil1984,Neeb,Hideki}, avoids the topology of $\mathcal{L}(\mathbb{E},\mathbb{F})$, using the product topology of $\mathbb{E}\times \mathbb{F}$ instead.
\begin{defn} Let  $f: U\subseteq \E\rightarrow\F$ be a  mapping between Fr\'{e}chet spaces, where $U$ is an open subset in $\E.$ We say that $f$ is  \textit{G\^ateaux differentiable} at $x\in U$ in the direction $h\in\E$ if the following limit exists
	$$\partial_x f(h):=\lim_{t\rightarrow 0}\frac{f(x+th)-f(x)}{t}.$$
	We say that $f$ is $C^1$-G\^ateaux differentiable  on $U$ if $f$ is continuous, the limit exists for all $x\in U$ and $h\in \E,$ and $\partial f: U\times \E\rightarrow \F$ is continuous relative to the product topology.
	Inductively we define the $C^k$-G\^ateaux differentiable mappings for $k\geq 2,$ and the G\^ateaux smooth mappings.
\end{defn}

This notion of differentiability  is weaker even in the context of Banach spaces, where  $C^{k+1}$-G\^ateaux differentiability implies $C^k$-Fr\'{e}chet differentiability, but the classes of smooth mappings coincide. 
There is no need for the spaces to be Fr\'echet, one can use locally convex topological vector spaces in the above definition, but we are focused on our goal: the smooth Fr\'{e}chet manifolds. This concept of $C^k$-G\^ateaux differentiable mappings coincides with those of $C^1_{MB}$-mappings in the sense of Michal-Bastiani  \cite{Bastiani,Michal} or $C^k_c$-mappings in the sense of Keller  \cite{Keller}. Using this differentiability concept one can introduce a smooth Fr\'{e}chet manifold, as in \cite{Ham1982}.
\begin{defn} A smooth \textit{Fr\'{e}chet manifold} is a Hausdorff topological space with an atlas of coordinate charts taking their value in Fr\'{e}chet spaces, such that the coordinate transition functions are  all G\^ateaux smooth mappings between Fr\'{e}chet spaces.
\end{defn}
Although this definition is the most popular one, it raises serious barriers when one tries to define some elementary geometric objects, e.g. differential forms. Of course, there are  attempts in this field to  use a stronger notion of differentiability, see \cite{Muller} for example, but most of them seem to fail in having serious applications in  infinite dimensional differential geometry. To be able to do some decent analysis one has to consider smooth  Fr\'{e}chet manifolds as particular cases of a more general notion: the \textit{smooth convenient manifolds}.

J. Boman had in \cite{Boman}  the idea to test the smoothness along smooth curves: a mapping $f$ from $\mathbb{R}^d$ to  $\mathbb{R}$ is smooth if and only if it sends smooth curves $u\in C^{\infty}(\mathbb{R},\mathbb{R}^d)$ into smooth curves $f\circ u \in C^{\infty}(\mathbb{R},\mathbb{R}).$ This concept was extended to mappings between locally convex spaces by A. Fr\"olicher and A. Kriegl and it will agree in the case of Fr\'{e}chet spaces with most of the smoothness notions already defined there. 
One can consult \cite{Averbukh,Keller} for a comparison between different differentiability concepts for locally convex spaces.
Thus, in \cite{Frolicher}, the authors constructed the so called convenient calculus for  locally convex topological vector spaces. In this context the k-fold differentiability is defined directly as well as infinite differentiability and one can avoid the topology of the space $\mathcal{L}(\E,\F).$ The remaining  of this section is devoted to introducing the notion of convenient smoothness. We will explain why this notion coincides with the G\^ateaux smoothness in the case of Fr\'{e}chet manifolds.

For  locally convex topological vector space $\E$ we call the final topology with respect to all smooth curves $c\in C^{\infty}(\mathbb{R},\E),$  the \textit{$\cinf$-topology}.
\begin{defn} A subset $U\subseteq \E$ is called \textit{$\cinf$-open }  iff $c^{-1}(U)$ is open in $\mathbb{R}$ for all $c\in C^{\infty}(\mathbb{R},\E),$ and we denote by \textit{$\cinf \E$} the space $\E$ equiped with this topology. 
\end{defn}

In other words the $\cinf$-topology  is the finest topology on $\E$ such that all the smooth curves $c:\mathbb{R}\rightarrow \E$ become continuous. If $\E$ is a Fr\'{e}chet space then the  $\cinf$-topology coincides with the given locally convex topology, according to \cite{Michor}, but in general the $\cinf$-topology is finer than any locally convex topology with the same bounded sets. The space $\cinf \E$ is not a topological vector space in general.

In a locally convex space $\E$ a curve $c:\mathbb{R}\rightarrow\E$ is called smooth if all its derivatives exist and are continuous.  The smoothness of the curves does not depend on the topology given on $\E$, in the sense that for all topologies leading to the same dual we have the same family of smooth curves.
In fact it depends only on the family of bounded sets, the  \textit{bornology} of $\E,$  see the Appendix.

\begin{defn}Let $\E,\F$ be locally convex spaces, a mapping $f: U\subseteq \mathbb{E}\rightarrow\mathbb{F}$ defined on a $\cinf$-open subset $U$  it is called  \textit{convenient smooth} if  it maps smooth curves in $U$ into smooth curves in $\F.$
\end{defn}

With this concept of smoothness there exist convenient smooth mappings which are not continuous, but all the convenient smooth mappings are continuous relative to the $\cinf$-topology, according to \cite{Frolicher,Michor}. The G\^ateaux smoothness will imply convenient smoothness but not conversely. Anyway, on Fr\'{e}chet spaces the two notions coincide.

\begin{prop} Let $\E,\F$ be Fr\'{e}chet spaces and $U\subseteq \E$ a $\cinf$-open subset, then $U$ is open and the mapping $f:U\subseteq \mathbb{E}\rightarrow\mathbb{F}$ is G\^ateaux smooth if and only if it  is convenient smooth.
\end{prop}

\begin{proof} As we mentioned before in the case of a Fr\'{e}chet space $\cinf \E=\E $ and thus $U$ is open for the given topology on $\E.$ If $f$ is G\^ateaux smooth  then one can  easily see that $f\circ c$ will be G\^ateaux smooth for all smooth curves $c\in C^{\infty}(\mathbb{R},\E)$, thus $f$ is convenient smooth. We denote by $d_xf$ the derivative of a convenient smooth mapping as in the Appendix.
	If $f$ is smooth in the convenient sense then, by Proposition \ref{derivata} in the Appendix, the mapping  $df: U\rightarrow \mathrm{L}(\mathbb{E},\mathbb{F})$ exists and is convenient smooth. Here $\mathrm{L}(\mathbb{E},\mathbb{F})$ denotes the space of bounded (not necessarily continuous) linear mappings between convenient vector spaces. See the Apppendix for the definition of a convenient vector space. The cartesian closedness property, Proposition \ref{cartesian  closedness}, implies that $\partial f:=df: U\times\E\rightarrow\F$ is convenient smooth, thus continuous relative to the $\cinf$- topologies, which coincide here with the given topologies on $\E, \F.$
\end{proof}

\begin{rem} This notion of convenient smoothness or Boman smoothness can substitute the  most used notion of smoothness for Fr\'{e}chet spaces and implicitly for smooth Fr\'{e}chet manifolds. 
\end{rem}

\section{Convenient manifolds.  Regular Lie groups} Let us introduce now the geometric objects we will need. In order to do that we follow closely \cite{Michor}:

\begin{defn} (Convenient manifolds) A chart $(\mathcal U,\f)$ on a set $M$ is a bijection $\f : \mathcal U\rightarrow U\subseteq \E_{\mathcal U}$ from a subset $\mathcal U\subseteq M$ onto a $\cinf$-open subset $U$ of a convenient vector space $\E_{\mathcal U}.$
	For two charts $(\mathcal U_{\alpha},\f_{\alpha})$ and $(\mathcal U_{\beta},\f_{\beta})$ on M the mapping 
	$$\f_{\raisebox{-2pt}{\tiny $\alpha\beta$}}:= \f_{\alpha}\circ \f^{-1}_{\beta} : \f_{\beta}(\mathcal U_{\alpha} \cap \mathcal U_{ \beta})\rightarrow \f_{\alpha}( \mathcal U_{\alpha} \cap \mathcal U_{ \beta})$$
	is called the transition mapping. A family $(\mathcal U_{\alpha},\f_{\alpha})_{\alpha\in A}$ of charts on $M$ is called an \textit{atlas} for $M,$ if the sets $\mathcal{\alpha}$ form a  cover of $M$ and all transition mappings are  defined on $\cinf$-open subsets.
	
	An atlas for $M$ is called smooth if all transition mappings $\f_{\alpha\beta}$ are \textit{convenient smooth}. Two smooth atlases are called smooth-equivalent if  their union is again a smooth atlas. An equivalence class of smooth atlases is a \textit{smooth structure} for $M.$ 
	A \textit{smooth convenient manifold} $M$ is a set together with a smooth structure on it.
\end{defn}

The isomorphism type of the modelling spaces $\E_{\mathcal U}$ is constant on the connected components of the manifold $M,$ since the derivative of the chart changings are linear isomorphisms. The manifold $\DiffS,$ considered as an example in this article, is a connected manifold. Since we are focused only in offering a theoretical background for the Euler-Poincar\'e-Arnold equations, we are entitled to consider $\E_{\mathcal U}=\E$  in some of our reasonings to avoid further technicalities.
In the case of manifolds modelled on Fr\'{e}chet spaces the above definition coincides with the one presented before.

\begin{defn} A mapping $f:M\rightarrow N$ between  convenient smooth   manifolds is called \textit{convenient smooth} if for each $p\in M$ and each chart $(\mathcal V_{\beta}, \psi_{\beta})$ on $N,$ with $f(p)\in \mathcal V_{\beta}$ there is a chart $(\mathcal U_{\alpha}, \f_{\alpha})$ on $M$ with $p\in \mathcal U_{\alpha},$ $f(\mathcal U_{\alpha})\subseteq\mathcal V_{\beta}$ and \textit{the local representative} $\psi_{\beta}\circ f\circ \f^{-1}_{\alpha}$  is convenient smooth. This is the case if and only if $f\circ c$ is a  smooth curve on $M$ for each smooth curve $c:\RR\rightarrow M.$ 
\end{defn}

\begin{defn} A convenient smooth Lie group $G$ is a convenient smooth manifold and a group such that the multiplication
	$m_G:G\times G\rightarrow G$ and the inversion $i_G:G\rightarrow G$ are convenient smooth.
\end{defn}

The conjugation mapping $c_g(x):=gxg^{-1}$ generates the adjoint representation of the Lie group $G$ 
$$\Ad: G\rightarrow GL(\mathfrak{g})\subset L(\mathfrak{g}, \mathfrak{g}),$$ considered as a convenient smooth mapping into  $L(\mathfrak{g}, \mathfrak{g}).$ 
In this way it makes sense to define the ajoint representation of the Lie algebra $\mathfrak{g}$ as
$$\ad:=T_e\Ad :\mathfrak{g}\rightarrow \mathfrak{gl}(\mathfrak{g})= L(\mathfrak{g}, \mathfrak{g}).$$
In the particular case $G=\DiffS$, the group of orientation-preserving diffeomorphisms of the circle,  its Lie algebra is $\mathfrak{g}=\CS$ the vector space of smooth real functions on the circle.  The adjoint representation of $\mathfrak g=\CS$ is given by
\begin{equation}
\label{adDiff}
\ad_uv=-[u,v]=u_xv-v_xu, \ \ u,v\in \CS.
\end{equation}

For an infinite dimensional Lie group the Lie exponential mapping may not exist or may not be smooth. An attempt to find a condition which ensures both these properties led to the notion of regular Lie groups, introduced by J. Milnor \cite{Mil1984}.

\begin{defn}\label{regular} A convenient smooth Lie group is called \textit{regular} if for every curve $u\in\mathrm{C}^{\infty}(\mathbb{R},\mathfrak{g})$ there exists a curve $g\in\mathrm{C}^{\infty}(\mathbb{R},G)$ such that
	
	$$\begin{cases} g(0)=e_{\raisebox{-2pt}{\tiny G}},  & \\  
	R_{g(t)^{-1}}\dot{g}(t)=u(t).& \end{cases}$$
	and the evolution mapping
	$$\mathrm{evol}^r_{\raisebox{-2pt}{\tiny G}} : \mathrm{C}^{\infty}(\mathbb{R},\mathfrak{g})\rightarrow G,\hspace{0,5cm} \mathrm{evol}^r_{\raisebox{-2pt}{\tiny G}}(u):=g(1) $$
	exists and is convenient smooth.
\end{defn}

Examples of regular Lie groups are offered by the \textit{strong ILH-Lie groups} in the sense of H. Omori \cite{Hideki}.
\begin{defn}
	A topological group $G$ is a strong ILH-Lie group modelled on the Fr\'echet space $\mathfrak g:=\E=\lim\limits_{\leftarrow q} \E^q, q\geq d,$ if and only if there exists a system $ G^q, q\geq d,$ of topological groups $G^q$ satisfying the following conditions:
	\begin{itemize}
		\item [($G_1$)] every group $G^q$ is a Hilbert manifold modelled on the Hilbert space $\E^q,$
		\item  [($G_2$)] $G^{q+1}$ is a dense subgroup in $G^q,$ and the embedding $G^{q+1}\subset G^q$ is a mapping of class $C^{\infty}$,
		\item  [($G_3$)] $G=\underset{q\geq d}{\cap} G^q$ with inverse limit topology,
		\item [($G_4$)] the group multiplication mapping $m_G:G\times G\rightarrow G$ extends to a mapping $G^{q+l}\times G^q\rightarrow G^q$ of class $C^l,$
		\item [($G_5$)] the inversion mapping $i_G:G\rightarrow G$ extends to a mapping $G^{q+l}\rightarrow G^q$ of class $C^l,$
		\item [($G_6$)] for each $g\in G^q$ the right translation $r_{g}:G^q\rightarrow G^q$ is a mapping of class $C^{\infty},$
		\item [($G_7$)] let $\mathfrak{g}^q:=\E^q$ be the tangent space of $G^q$ at the identity $e\in G^q,$ and let $TG^q$ be the tangent bundle. The mapping $Tr:\mathfrak{g}^{q+l}\times G^q\rightarrow TG^q$ defined by $Tr(u,g):=R_gu$ is a mapping of class $C^l,$ 
		\item [($G_8$)] there exists an open neighborhood $U$ of zero in $\mathfrak{g}^d$ and a diffeomorphism $\Phi$ of $U$ onto an open neighborhood $\mathcal{U}$ of the unity $e\in G^d,$  $\Phi(0)=e,$ such that the restriction of $\Phi$ to $U\cap \mathfrak{g}^q$ is a diffeomorphism of the open subset $U\cap \mathfrak{g}^q$ from $\mathfrak{g}^q$ onto an open subset $\mathcal{U}\cap G^q$ from $G^q$ for any $q\geq d.$
	\end{itemize}
	\end{defn}
	
In the case $G=\DiffS$ for $q > \frac{3}{2}$, the set $G^q:=\D{q}$ of $C^1$-orientation preserving diffeomorphisms
of the circle, which are of class $H^q$, has the structure of a Hilbert manifold modelled on $\mathbb E^q:=H^q(\mathbb S^1)$. It is also a topological group.  $\DiffS=\cap_{q >\frac{3}{2}} \D{q}$ is a strong ILH-Lie group according to \cite{Hideki}, thus a regular convenient smooth  Lie group.
	
\section{Euler-Arnold equations on  regular  Lie groups}\label{Euler-Arnold on regular groups}
In oder to define a  Riemannian metric on a  regular convenient Lie group $G$  an inner product on the Lie algebra $\mathfrak{g}$ is extended to every tangent space by right translations
\begin{equation}\label{metric}\langle u_g,v_g \rangle_{g}=\langle R_{g^{-1}}u_g,R_{g^{-1}}v_g\rangle_e,\hspace{1,5cm}u_g,v_g\in T_gG,\hspace{0,2cm} g\in G. \end{equation}
If this inner product is generated by an isomorphism $A: \mathfrak{g}\rightarrow\mathfrak{g}^*$, which is positive-definite and symmetric with respect  to the natural pairing $(\cdot,\cdot)$ between elements of  $\mathfrak{g}^*$ and  $\mathfrak{g}$
\begin{equation}\label{notations}
\langle u,v\rangle^A_e:=(u,Av)=(Au,v),\hspace{1cm}  u,v\in \mathfrak{g},
\end{equation}
then this operator is called the \textit{inertia operator} on $G$.
The natural pairing is actually the evaluation mapping. By Remark \ref{naturalpairing} in the Appendix,  it is  convenient smooth if, for example, the topological dual $g^*$ is endowed with the strong topology and $g$ is a convenient vector space. It will never be G\^ateaux smooth because  G\^ateaux smoothness implies continuity.

\begin{rem}
	
When working with an infinite dimensional Lie group one can not consider  bi-invariant metrics because in this case  the Riemannian exponential mapping and the Lie exponential mapping will coincide and the latter one can behave bizarrely. Another problem occurs, exemplified here for the Fr\'{e}chet \hspace{-0,1cm}-Lie group $G=\DiffS$. In order to maintain the isomorphism property of the inertia operator, described above, we have to restrict $\mathfrak{g}^*$ to its regular dual
$$\mathfrak{g}^*_{reg}\cong\CS,$$
defined as the space of linear functionals of the form
$$u\rightarrow \int_{\Circle} m\cdot u \hspace{0,05cm} dx, $$ for $m\in\CS,$  due to \cite{Kirillov,Kolev}. The pairing between the elements of   $\mathfrak{g}^*_{reg}$ and $\mathfrak{g}$ will be given by the $L^2(\mathbb{S}^1)$-inner product
\begin{equation}\label{pairing}(u,v):=\langle u,v\rangle_{L^2(\Circle)}=\int_{\Circle} u\cdot v \hspace{0,1cm} dx.\end{equation}
The topology on $\mathfrak{g}^*_{reg}$ is not the induced one and now the pairing becomes even G\^ateaux smooth, which is impossible without the above convention. With this convention the inertia operator $A:\mathfrak{g}\rightarrow \mathfrak{g}^*_{reg}$ will be called \textit{regular inertia operator}.
\end{rem}

If the adjoint of $\ad_v$ relative to the inner product \eqref{notations} exists then the geodesics can be determined with the help of this operator. We remind here that a bilinear operator is bounded if and only if it is convenient smooth by Theorem \ref{derivata}. On $\CS$ boundedness will be equivalent with continuity, being a bornological space.

\begin{thm}\label{Arnold} (V. Arnold,  \cite{Arn1966})
		If the inner product $\langle \cdot,\cdot\rangle_e:\mathfrak{g}\times\mathfrak{g}\rightarrow\mathbb{R}$ defined in \eqref{notations} is bounded and there exists a bounded bilinear operator
		$$B :\mathfrak{g}\times\mathfrak{g}\rightarrow \mathfrak{g} ,$$ 
		 with the property
		$$\langle B(u,v),w \rangle_e= \langle u,\ad_v w \rangle_e,\hspace{0,5cm}w\in\mathfrak{g},$$ where $ad_v$ is the adjoint representation of $\mathfrak{g},$ 
		then a smooth curve $g(t)$ on the regular convenient Lie group $G$ is a geodesic for the right-invariant metric  defined by \eqref{metric}  if and only if its Eulerian velocity  $u(t)=R_{g(t)^{-1}}\dot{g}(t)$  satisfies the first order equation $$u_t=-B(u,u).$$ 
	\end{thm}
	\begin{proof} We present the proof adapted to the convenient approach, whereas  the original proof of V. Arnold is not rigorous enough in the infinite dimensional setting. Parts of the proof are presented in Section 46.4 of \cite{Michor}. 
		
	For a smooth $g: [a,b]\rightarrow G$  the energy functional of the curve $g$ is
		
		$$E(g):=\frac{1}{2}\int_a^b\langle \dot{g}(t),\dot{g}(t)\rangle_{g(t)}dt= \frac{1}{2}\int_a^b\langle g^*\kappa^r(\partial_t),g^*\kappa^r(\partial_t)\hspace{0,1cm}\rangle_edt,$$
		where we pullback the Maurer-Cartan form $\kappa^r$ on $T\RR,$ by $g.$

	The first part uses arguments borrowed from the finite dimensional case. One has to introduce  $g(s,t)$ a smooth variation  of the curve $g,$  $s\in (-\varepsilon, \varepsilon), t\in[a,b],$  with fixed endpoints $g(s,a)=g(a),$ $g(s,a)=g(b)$  and let us  denote the curves $u(s,t):= R_{g(s,t)^{-1}}\partial_t g(s,t)$ and $v(s,t):=  R_{g(s,t)^{-1}}\partial_s g(s,t).$ In particular we have $u_0(t):=u(0,t):[a,b]\rightarrow \mathfrak{g}$ and  $v_0(t):=v(0,t):[a,b]\rightarrow \mathfrak{g}.$ 
			\begin{equation*}\begin{split}\partial_s E(g)&= \frac{1}{2}\int_a^b2\langle \partial_s( g^*\kappa^r(\partial_t)),g^*\kappa^r(\partial_t)\hspace{0,1cm}\rangle_e \mbox{ }dt\\
		&= \int_a^b\langle \partial_t( g^*\kappa^r(\partial_s))-d(g^*\kappa^r )(\partial_t,\partial_s), g^*\kappa^r(\partial_t)\hspace{0,1cm}\rangle_e \mbox{ }dt
		\end{split}
		\end{equation*}
		because $[\partial_t,\partial_s]=0$ by Schwarz's theorem. Further
		
	\begin{equation*}\begin{split} & \int_a^b-\langle  g^*\kappa^r(\partial_s), \partial_t(g^*\kappa^r(\partial_t))\rangle_e- \langle [g^*\kappa^r (\partial_t),g^*\kappa^r (\partial_s)], g^*\kappa^r(\partial_t)\hspace{0,1cm}\rangle_e \mbox{ }dt\\
		&= -\int_a^b\langle  g^*\kappa^r(\partial_s), \partial_t(g^*\kappa^r(\partial_t))+  B(g^*\kappa^r (\partial_t), g^*\kappa^r (\partial_t)) \hspace{0,1cm}\rangle_e \mbox{ }dt,
		\end{split}
		\end{equation*}
		exploiting the fixed endpoints of the variation and applying the right Maurer-Cartan equation.
		
		The curve $g$ is a geodesic for the metric \eqref{metric} iff the derivative vanishes at $s=0$ for all variations $g(s,t)$ of $g$ with fixed endpoints. By Corollary 38.13 in \cite{Michor} the group $\mathrm{C}^{\infty}(\mathbb{R}, G)$ is a  regular convenient Lie group, if $G$ is a regular convenient  Lie group. It has the Lie algebra $\mathrm{C}^{\infty}(\mathbb{R}, \mathfrak{g})$ with the bracket $[X,Y](t):=[X(t),Y(t)]_{\mathfrak{g}}.$ Thus, following the definition of a  regular convenient  Lie group, for every curve
		$$v(s,t)\in \mathrm{C}^{\infty}(\mathbb{R},\mathrm{C}^{\infty}(\mathbb{R}, \mathfrak{g}))=\mathrm{C}^{\infty}(\mathbb{R}\times \mathbb{R} , \mathfrak{g})$$ there exists a curve
	
		$$g(s,t)\in \mathrm{C}^{\infty}(\mathbb{R},\mathrm{C}^{\infty}(\mathbb{R}, G))\subseteq \mathrm{C}^{\infty}(\mathbb{R}\times \mathbb{R} , G),$$
		such that
		$$\begin{cases}g(0,t)=g(t)\in \mathrm{C}^{\infty}(\mathbb{R}, G), & \\
		v(s,t):=  R_{g(s,t)^{-1}}\partial_s g(s,t).& \end{cases}$$
		
		In particular, every smooth curve $v_0:[a,b]\rightarrow \mathfrak{g}$ corresponds to a variation with fixed endpoints of $g.$ Eventually, for all $v_0$
		$$\int_a^b\langle v_0(t),\dot{u}_0(t)+B(u_0(t),u_0(t))\rangle_e \mbox{ }dt=0.$$
		Applying this identity for the smooth curve $v_0(t):= \dot{u}_0(t)+B(u_0(t),u_0(t))$ we get the conclusion, since the inner product \eqref{notations} is positive definite and smoothness implies continuity for curves.
	\end{proof}
		
In the case $G=\DiffS$, for the right-invariant metric induced as in \eqref{metric} and \eqref{notations} by the inner product $\langle u,v\rangle=\int_{S^1}Au\cdot v\  dx$, with $A:\CS\to \CS$ continuous linear, invertible, positive definite and $L^2$-symmetric, one gets
$$B(u,v)=A^{-1}\left\{ u\cdot (Av)_x+2 Av\cdot u_x     \right\},$$
since $\ad_uv$ has the expression given in \eqref{adDiff}.
Hence a curve $\f(t)$ is a geodesic of the right-invariant metric induced by the inertia operator $A$ if and only if its Eulerian velocity $u(t):=R_{\f(t)^{-1}}\dot{\f}(t)=\dot {\f}(t)\circ \f(t)^{-1}$ satisfies the equation
$$u_t=-A^{-1}\left\{ u\cdot (Au)_x+2 Au\cdot u_x     \right\}.$$

	The equation from the above theorem is called the \hspace{-0,2cm}\textit{ Euler-Arnold equation} induced by an inertia operator $A$. In general a Levi-Civita connection related to the Riemannian metric \eqref{metric} is not granted, because a metric like \eqref{metric}  generates  a flat mapping $v_g\mapsto\langle v_g, \cdot\rangle_g$ that is only injective. If the adjoint $\ad^T_u$ exists such a connection also exists. 
	To derive its formula we have to identify the space $\Gamma(TG)$ of smooth sections with the convenient vector space $C^\infty(G,\mathfrak g)$ of $\mathfrak g$-valued fuctions on $G$.  The isomorphism is induced by the right trivialization $\rho=(\pi_G,\kappa^r) :TG\rightarrow G\times\mathfrak{g}.$

\begin{prop}\label{bornologicaliso} The mapping  $\mathcal R : C^{\infty}(G,\mathfrak{g})\rightarrow \Gamma(TG)$
	 $$\widebar X\to \mathcal R_{\widebar X}$$ 
	 where $\mathcal R_{\widebar X}$ is defined by $\mathcal R_{\widebar{X}}(g):=R_g(\widebar X(g)),$ for $\widebar X\in \mathrm{C}^{\infty}(G,\mathfrak{g}),$ $g\in G,$ is a bornological isomorphism.
	\end{prop}	
\begin{proof} We have to argue that $\mathcal R$ and $\mathcal R^{-1}$ send bounded sets into bounded sets. Of course, we refer here to the von Neumann bornology corresponding to the natural topologies on $ C^{\infty}(G,\mathfrak{g})$ and $ \Gamma(TG)$, presented in the Appendix. For $\mathcal R^{-1}$ one gets the conclusion observing that $\mathcal R^{-1}(\xi)=\kappa^r( \xi) ,$ for $\xi\in\Gamma(TG),$ and the insertion mapping $i:\Gamma(TG)\times \Omega^1(G,\mathfrak g)\to \Omega^0(G,\mathfrak g):=C^\infty(G,\mathfrak g)$ is convenient smooth. 
	
	Since $TG$ is trivializable, let us denote by $f_\alpha: \Gamma(TG)\to C^\infty(\RR,\mathfrak g),$ the mappings  with respect to which $\Gamma(TG)$ has the initial topology (see the Appendix B) 
	$$f_\alpha:=c^*\circ (u_\alpha^{-1})^*\circ (pr_2\circ \rho)_*\circ (i_{\mathcal U_\alpha})^*.$$
	 Let us also denote by $g_\alpha:C^\infty(G,\mathfrak g)\to C^\infty(\RR,\mathfrak g)$  the mappings that give the topology on $C^\infty(G,\mathfrak g)$. Then $\mathcal R$ is bounded if and only if $f_\alpha\circ \mathcal R$ is bounded, since the von Neumann bornology coincides with the initial bornology induced by $f_\alpha$ on $\Gamma(TG)$ (cf. \cite{Gach}). But $f_\alpha\circ \mathcal R=g_\alpha \circ Id_{C^\infty(G,\mathfrak g)}$, thus $\mathcal R$ is bounded.
		
\end{proof}

	As a consequence $\mathcal R$ and $\mathcal R^{-1}$ are convenient smooth. In order to avoid some cumbersome expressions we  will use, in the next proposition, the same notation for the vector field and the curve attached via the isomorphism $\mathcal R.$ Relation \eqref{nablaR}  should be  written as $\nabla _{\mathcal R_{\widebar  X}}\mathcal R_{\widebar Y}:=\mathcal R_{\widebar{\nabla_X Y}},$ for some mappings $ \widebar X,\widebar Y, \overline{\nabla_X Y }\in \mathrm{C}^{\infty}(G,\mathfrak{g})$. Because the flat mapping is only injective the idea is to build the below candidate and to prove that this is the unique Levi-Civita connection, related to the right-invariant metric  \eqref{metric}.

\begin{prop}\label{metricconnection} Assume that for all $u\in\mathfrak{g}$ the adjoint $\ad^T_u$ with respect to the bounded inner product $\langle\cdot,\cdot \rangle_e$ exists and that $u\mapsto \ad_u^T$ is bounded. Then the Levi-Civita connection related to the metric \eqref{metric} exists and is given by
			\begin{equation}\label{nablaR}
			\nabla _{\mathcal R_{ X}}\mathcal R_{Y}:=\mathcal R_{\nabla_X Y},\hspace{0,5cm} \nabla_X Y \in \mathrm{C}^{\infty}(G,\mathfrak{g}),\end{equation}
			where
		\begin{equation}\label {R_X}
	\nabla_X Y(g):=T_g Y(\mathcal R_X(g))-\frac{1}{2}\ad_{X(g)}Y(g)+\frac{1}{2}\ad^T_{X(g)}Y(g)+\frac{1}{2}\ad^T_{Y(g)}X(g), \end{equation}
			for $X,Y\in \mathrm{C}^{\infty}(G,\mathfrak{g}),$ $g\in G.$
		\end{prop}
		
		\begin{proof} See Section 46.5 in \cite{Michor}.\end{proof}

		\begin{rem} This formula coincides pointwise with the corrected version of the formula given  in \cite{Constantin}, for the particular case $G=\DiffS$
			$$(\nabla_X Y)_{\f}=[X, Y-Y^R_{\f}]_{\f}+\frac{1}{2}\left([X^R_{\f},Y_{\f}^R]_{\f}+B(X^R_{\f},Y^R_{\f})_{\f}+B(Y^R_{\f},X^R_{\f})_{\f}\right),$$
			where, for $X\in \Gamma(\mathrm T\DiffS),$ the term $X_{\f}^R$ denotes the right-invariant vector field whose value at $\f\in \DiffS$ is $X_{\f}$, thus $X_{\f}^R=\mathcal R_{X_{\f}\circ \f^{-1}}$ and the operator $B(u,v):=\ad^T_vu$ was extended to the family of right-invariant vector fields by $B(Z,W)_{\f}=B(Z_{id},W_{id})\circ\f.$  
		\end{rem}

\section{The geodesic spray} \label{Sec:6}

An interesting phenomenon concerning the Euler-Poincar\'e-Arnold equations occurs for some Fr\'echet-Lie groups:   the propagator of the  evolution equation which describes the geodesic flow (Lagrangian coordinates) has better properties than the one corresponding to the  Euler-Arnold equation (Eulerian coordinates).  

We will exemplify this remark by considering  the case $G=\DiffS$, with a right-invariant metric constructed as in \eqref{notations} and  A a pseudo-differential operator of H\"ormander class $S^r_{1,0}$,  $r\geq 1$, as in \cite{Cismas2016}. The spray equation will be
\begin{equation}\label{sprayeq}\begin{cases} \f_t=v\\
v_t=S_{\f}(v) \end{cases},
\end{equation}
where
\begin{equation*}S_{\f}(v)=R_{\f}\circ S\circ R_{\f^{-1}}(v),\end{equation*}
and
\begin{equation*}S(u)=A^{-1}\{[A,u]D(u)+u[A,D](u)-2A(u)D(u)\}.
\end{equation*}
with $D:=\frac{d}{dx}$ on $S^1$, defined as in \cite{GR2007}. 

The mapping $\mathrm F:\mathrm T\DiffS\rightarrow \mathrm{TT}\DiffS,$ locally defined by
\begin{equation}\label{spraystandard}F(\f,v):=(\f,v,v,S_{\f}(v)),\end{equation}
extends to a smooth mapping  $F:T\D{q}\rightarrow TT\D{q}$, mostly because of the commutators that appear in the above expressions. Hence, it is possible to recast the spray equation as an ODE on suitable Hilbert spaces and to work on the Hilbert approximations of the ILH Lie group $ \DiffS.$ 
This phenomenon is  exploited in order to obtain the existence of an  integral curve of the geodesic spray, see \cite{BauerEscherKolev, Cismas2016, Constantin, EM1970, EKHS, Kohlmann}.

On the other hand, denoting $\f_t=u\circ \f$, one discovers via Theorem \ref{Arnold} that  $\f$ satisfies the spray equation iff the Eulerian velocity $u$ satisfies the the  Euler-Arnold equation
$$u_t= -A^{-1}\{ u \cdot (Au)_x+2 Au \cdot u_x\}.$$ 
Unfortunately, this equation has a derivative loss when considered on the Hilbert approximations, thus it needs some Nash-Moser schemes in order to be solved in the smooth category.  Such an approach is used in \cite{Cismas2019}, with a comparison between these two methods.

The above discussion contains the essence of the so called "geometric method in hydrodynamics"\cite{Escher-Kolev}. Most of the papers concerned with this geometric method do not define rigorously the  spray in the case of a Fr\'{e}chet manifold. This is what we intend to do in the sequel. We are interested here only in sprays related to a right-invariant metric on a regular Lie group. 
There is not a direct way to define a spray  on a Fr\'{e}chet \hspace{-0,1cm}-Lie group using the concept of G\^ateaux smoothness, without introducing some conventions or losing generality.  The convenient calculus permits us to define a geodesic spray on Fr\'{e}chet manifolds and afterwards one can use the ILH structure of the Lie group for proving the existence of an integral curve. 

\subsection{Sprays on Banach manifolds}
In the case of a Banach manifold $M$ modelled on a Banach space $\mathbb{E}$, in order to define the spray related to a metric, the classical approach is to use  the flat mapping $$\widehat{g}:T_pM\rightarrow T_p^*M,\hspace{1cm}\xi\rightarrow g(p)(\xi, \cdot),$$
where $g$ is a smooth Riemannian metric on $M.$ On the cotangent bundle of $M$ we can define the canonical Liouville 1-form by
$$\Theta_{\omega}(X):=\omega(T\pi^* (X)), \hspace{1cm}\omega \in T_x^*M, X\in T_{\omega}(T^*M), $$
where $\pi^* :T^*M\rightarrow M$ is the canonical projection. There is also a canonical symplectic form on $T^*M$ obtained as
 $$\Omega=-d \Theta,$$ 
 where $d$ is the exterior derivative of a 1-form.

We can pullback the Liouville form by the flat mapping  $\widehat{g}$ to obtain a 1-form $\Theta^g$ on $TM$
$$\Theta^g_{\xi}(X):= g(p)(\xi,T\pi_M(X)), \hspace{1cm}\xi\in TM, X\in T_{\xi}(TM),$$
and further a symplectic form on $TM$
$$\Omega^g:=-d\Theta^g.$$

If the metric is strong we can associate to every function $H$ on $TM$ a Hamiltonian vector field $F_H$ on $TM$ defined as
\begin{equation}\label{Hamiltonian} d_{\xi}H(X):=\Omega^g(F_H(\xi),X),
\end{equation} where $\xi\in TM$ and $ X\in T_{\xi}(TM).$
If the metric is weak the flat mapping and the symplectic form $\Omega^g$ are only injective and thus given a function $H$ on $TM$ the Hamiltonian vector field corresponding to it may not exist, but if exists it is given by the above relation \eqref{Hamiltonian}.
\begin{defn}(see \cite{Lang99}) The geodesic spray $F$ associated to a strong  metric $g$ is defined as the Hamiltonian vector field of the energy function
	$$E(\xi):=\frac{1}{2} g(p)(\xi,\xi).$$
\end{defn}

In a local chart $U_{\mathbb{E}}\times\mathbb{E}$ of $TM$ the Hamiltonian vector field $F$ is
$$F(x,v):=(x,v,v, S(x,v)),$$ where $S(x,v)$ is defined by
$$g(x)(S(x,v),u)= \frac{1}{2} D_xg (u) (v,v)-D_x g(v)(v,u),$$
for $x\in U_{\mathbb{E}}\subseteq \mathbb{E} $ and $u,v\in\mathbb{E},$ where $D_xg$ represents the Fr\'{e}chet derivative of the local representative of the metric.
Since the flat mapping $\widehat{g}$ is bijective one gets
$$S(x,v)=pr_2[ \widehat{g}^{-1}(x, \frac{1}{2} D_xg (u) (v,v)-D_x g(v)(v,u))].$$

\subsection{Sprays for the geometric method}

 By $TM$ we will denote the kinematic tangent bundle of an infinite dimensional manifold, modeled on a convenient vector space $\mathbb E.$

\begin{defn} We define a spray $F$ to be a convenient smooth section of both $ \pi_ {TM}:TTM\rightarrow TM$ and $T\pi_M :TTM\rightarrow TM$ (symmetric vector field) which satisfies the quadratic condition
	$$F\circ m_{\lambda}^{TM}=Tm_{\lambda}^{TM}\circ m_{\lambda}^{TTM}\circ F,$$
	where  $m_{\lambda}^{TM}, m_{\lambda}^{TTM}$ denote the fiber scalar multiplications.\end{defn}

Let $G$ be a  regular convenient Lie group and $g$ a convenient smooth right-invariant metric defined as in \eqref{metric} by a bounded inner product $\langle \cdot, \cdot \rangle$. Generally, a Riemannian metric $g$ on an convenient manifold $G$ can be defined as a convenient smooth section of the vector bundle $L(TG\oplus TG,G\times\RR)$, that gives a positive definite and symmetric bilinear form $g(p)(\cdot,\cdot)$ on each tangent space $T_pG,$ $p\in G.$
Let us consider the following mapping on $TTG$
$$\Theta^g_\xi(X):=g(\xi,T_\xi\pi_G(X))=\langle \kappa^r(\xi), (\pi_G^*\kappa^r)_\xi(X)\rangle,\hspace{0,2cm}X\in T_\xi TG,$$
where   $\langle\cdot,\cdot\rangle$ is the inner product \eqref{notations},  $\kappa^r$  the right Maurer-Cartan form and the mapping $\pi_G: TG\to G$ is the canonical projection.   $\Theta^g$ is a convenient smooth section of the convenient convenient smooth vector bundle $L(TTG,TG\times\RR)$, thus  a kinematic 1-form on the kinematic tangent bundle $TG,$ when $G$ is a  regular convenient Lie group.  The smoothness of $\xi\to \kappa^r(\xi)$ from $\Gamma(TG)$ to $C^\infty(G,\mathfrak g)$ was justified in the previous section. Now, one can define a kinematic 2-form on $TG$ by
$$\omega^g(Y,X):=-d\Theta^g(Y,X).$$
To every right-invariant metric on a regular  convenient Lie group $G$ one can associate the energy function $E:TG\rightarrow \RR$
$$E(\xi):=\frac{1}{2}g(\xi,\xi)=\langle \kappa^r(\xi), \kappa^r(\xi)\rangle, \hspace{0,5cm}\xi\in TG.$$

\begin{prop} If there exists a  vector field $F$ on the kinematic tangent bundle $TG$ satisfying
	\begin{equation}\label{spray}i_F\omega^g=dE,\end{equation} then it is unique  and it is a right-invariant spray.
\end{prop}

\begin{proof}
	If $F(x,v)=(x,v, S_1(x,v), S_2(x,v))$ is the local representative of the  vector field $F$ then, using Remark \ref{remarca metrica} in the Appendix, the local expression of $i_F\omega^g=dE$ is
    $$d_xg(u)(v, S_1(x,v))+g(x)(w, S_1(x,v))-d_xg(S_1(x,v))(v,u)-g(x)(S_2(x,v),u)$$
	$$=\frac{1}{2}d_xg(u)(v,v)+g(x)(v,w),$$ for $X=(x,v,u,w).$
	Choosing $u=0$ implies $S_1(x,v)=v.$ Hence $F$ is a symmetric  vector field and $S_2(x,v)$ satisfies
	$$g(x)(S_2(x,v),u)=\frac{1}{2}d_xg(u)(v,v)-d_xg(v)(v,u).$$
	Finally $S_2(x,\lambda v)=\lambda^2 S_2(x,v)$ and $S_2$ is quadratic in $v,$ thus $F$ is a spray.
	
	The vector field (if exists) defined by \eqref{spray} has to be unique because $g$ is non-degenerate and one can easily see that $F$ is actually invariant under any isometry of the metric \eqref{metric} because $E$ and $\omega^g$ are.
\end{proof}

There are two possible trivializations of $TTG\cong T(G\circledS\mathfrak g)$, the first one is 
$$tr^1_{TTG}(\xi_g,\xi_u)=(g,u, \kappa^r(\xi_g), \xi_u+[u, \kappa^r(\xi_g)]), \ \ \ \ g\in G,  \xi_g\in T_gG, u\in \mathfrak g, \xi_u\in T_u\mathfrak g.$$
and is obtained decomposing  $T(G\circledS\mathfrak g)$ as the semidirect product between  $G\circledS \mathfrak g$ and $\mathfrak g\circledS\mathfrak g,$ see \cite{Esen-Gumral}. The second one is obtained using  $tr^2_{TTG}=\sigma\circ (\rho\times id)\circ T\rho$, where $\sigma (g,u,v,w)=(g,v,u,w)$ is the canonical involution. Thus $$tr^2_{TTG}(\xi_g,\xi_u)=(g,u,\kappa^r(\xi_g), \xi_u), \ \ \ \ \ \ \ \ g\in G, u\in \mathfrak g.$$
Of course, the symmetric vector fields on $TG$ are invariant under these two trivializations. We prove now that, when the Arnold operator exists, also a spray related to the right-invariant metric (\ref{metric}) exists. The formula is similar to the one obtained in \cite{Kouranbaeva} in a more restrictive setting. Propositions \ref{formula spray} and \ref{propspray} together with Theorem \ref{Arnold} and Proposition \ref{metricconnection}  form a unitary convenient approach of the so called geometric method in hydrodynamics.
\begin{prop}\label{formula spray} If the inner product \eqref{notations} is bounded and the Arnold operator $ad^T:\mathfrak{g}\times\mathfrak{g}\rightarrow\mathfrak{g}$
		exists and is convenient smooth then the mapping
		$$F: TG\rightarrow  TTG,$$
	defined, with respect to any of the two  trivializations of $TTG$, by
		\begin{equation}\label{spray existence formula}F(\xi):=(g, \kappa^r(\xi),  \kappa^r(\xi), -ad^T_{ \kappa^r(\xi)} \kappa^r(\xi) ), \hspace{0,5cm}\xi\in T_gG,\end{equation}
		is convenient smooth and satisfies the identity
		$$i_F\omega^g(X)=dE(X), \hspace{0,5cm}\forall \hspace{0,1cm} X\in TTG.$$
	\end{prop} 
	
	\begin{proof}  Let $\E$ be a convenient vector space and   $f: M\rightarrow \E $ a convenient  smooth  mapping.  Since, by \cite{Michor},  a kinematic vector field is also an operational vector field, then for a convenient smooth kinematic vector field $\xi $ one obtains a convenient smooth $\E$-valued mapping by $\xi(f):=pr_2\circ Tf\circ \xi.$ Here the right term is the differential of $f$, but in order to avoid another $d$ in our arguments we prefer this notation, compare to Section 28.15 in \cite{Michor}. The following formula holds
		$$\xi\left(\langle f, g \rangle \right)(p)= \langle \xi(f), g \rangle (p)+  \langle f, \xi(g) \rangle (p), \hspace{0,2cm}p\in M,$$
		for a bounded inner product $\langle \cdot, \cdot \rangle$ on $\E,$ and $f,g:M\to \E$ smooth.
		
	We will do the computations using the second trivialization of $TTG$. The kinematic vector field $F$ defined above is symmetric, hence we get $\pi_G^*\kappa^r(F(\xi))=\kappa^r(\xi).$
	The global formula holds for the exterior derivative of a kinematic differential form, according to Section 33.12 in \cite{Michor}, thus
			\begin{equation*}\begin{split}\omega^g_\xi(F(\xi), X(\xi))&=-F(\Theta^g\circ X)(\xi)+X(\Theta^g\circ F)(\xi)+(\Theta^g\circ [F,X])(\xi)\\
		&=-\langle F(\kappa^r)(\xi), \pi^*_G \kappa^r(X(\xi))\rangle-\langle \kappa^r(\xi), F(\pi^*_G \kappa^r(X))(\xi)\rangle\\
		&\quad +\langle X(\kappa^r)(\xi), \kappa^r(\xi)\rangle+\langle \kappa^r(\xi), X(\pi^*_G \kappa^r(F))(\xi)\rangle\\
		&\quad +\langle \kappa^r(\xi), \pi^*_G \kappa^r([F,X](\xi)) \rangle\\
		&=-\langle F(\kappa^r)(\xi), \pi_G^*\kappa^r(X(\xi))\rangle+\langle X(\kappa^r)(\xi), \kappa^r(\xi)\rangle \\
		&\quad - \langle \kappa^r(\xi), d(\pi_G^*\kappa^r)(F(\xi),X(\xi))\rangle.
		\end{split}
		\end{equation*}
		But
		\begin{equation*}\begin{split}
		d(\pi_G^*\kappa^r)(F,X)&=\pi_G^*\hspace{0,1cm} d\kappa^r(F,X)=\pi^*_G\left(\frac{1}{2}[\kappa^r,\kappa^r]_{\raisebox{-3pt}{\tiny $\wedge$}}(F,X)\right)\\
		&=[\kappa^r(T\pi_G(F)), \kappa^r(T\pi_G(X))]=\ad_{\kappa^r(T\pi_G(F))}\kappa^r(T\pi_G(X)),
		\end{split}\end{equation*}
		by applying the Maurer-Cartan equation.
		In  the same time  $T\pi_G(F(\xi))=\xi$ and $F(\kappa^r)(\xi)=pr_2\left( T_\xi\kappa^r(F(\xi))\right)=-\ad^T_{\kappa^r(\xi)}\kappa^r(\xi),$ since
		$\sigma\circ (\rho\times id)\circ(T\pi_G,T\kappa^r) $ gives the chosen trivialization of  $TTG.$
		Eventually
		\begin{equation*}\begin{split}\omega^g_\xi(F(\xi), X(\xi))=&-\langle -\ad^T_{\kappa^r(\xi)}\kappa^r(\xi), \kappa^r(T\pi_G(X(\xi))\rangle+\langle X(\kappa^r)(\xi), \kappa^r(\xi)\rangle
		\\&-\langle \ad^T_{\kappa^r(\xi)}\kappa^r(\xi), \kappa^r(T\pi_G(X(\xi))\rangle\\
		=&\langle X(\kappa^r)(\xi), \kappa^r(\xi)\rangle=d_\xi E(X(\xi)).
		\end{split}
		\end{equation*}
		\end{proof}
	Written in the manner of Proposition \ref{metricconnection} the above result becomes
	
	\begin{prop}\label{propspray}  Assume that for all $u\in\mathfrak{g}$ the adjoint $\ad^T_u$ with respect to the bounded inner product $\langle\cdot,\cdot \rangle_e$ exists and that $u\mapsto \ad_u^T$ is bounded. Then the metric spray is given by
		$$\displaystyle F=\mathcal R_{\widebar S},$$
		where $\mathcal R:C^{\infty}(TG,T\mathfrak g)\rightarrow  \Gamma(TTG)$ is the canonical isomorphism and the mapping $\widebar S:TG\rightarrow T\mathfrak g$ is defined as
		$$\displaystyle \widebar {S}(\xi)=\left(\kappa^r (\xi),-\ad^T_{\kappa^r(\xi)}\kappa^r(\xi)\right), \quad \xi\in TG.$$
		\end{prop}
		\begin{proof} One has only to argue that
			$$\displaystyle \left(\sigma\circ (\rho\times id)\circ T\rho \right)(\mathcal R_{\widebar S}(\xi))=\left(g,\kappa^r (\xi), \kappa^r(\xi), -\ad^T_{\kappa^r(\xi)}\kappa^r(\xi)\right)$$
		which is straightforward, since by its very definition
		$$\mathcal R_{\widebar S}(\xi)=R_\xi (\widebar S (\xi))=\left(\xi,\kappa^r(\xi), \ad^T_{\kappa^r(\xi)}\kappa^r(\xi)\right)$$
		\end{proof}

			\begin{prop}\label{spray-metric} If the inner product \eqref{notations} is bounded and the operator $ad^T$ exists and is bounded, then a  smooth curve $g:\RR\rightarrow G$ is a geodesic of the right-invariant metric \eqref{metric} if and only if $\dot{g}(t):\RR\rightarrow TG $  is an integral curve of the right-invariant spray $F$ defined by \eqref{spray}, and we call it the geodesic spray corresponding to this metric.
			\end{prop}
			\begin{proof} It is a straightforward consequence of Proposition \ref{Arnold} and of Proposition \ref{formula spray}, since $\kappa^r(\dot{g}(t))=R_{g(t)^{-1}}\dot{g}(t).$
			\end{proof}
			\begin{rem} 
				Since we  work on manifolds modelled on non normable locally convex spaces the existence of an integral curve for a smooth vector field is not granted. Most of the times we speculate the right invariance of the geodesic spray. It is a challenge to find new ways for proving the existence of  integral curves for a given spray,  motivated in practice by metrics which are not right-invariant. These kind of metrics can appear in fields like  image processing or shape analysis.
				
				 In order to prove the existence of an integral curve of the geodesic spray one has to visualize it in an appropriate way. The Lie group $G=\DiffS$ is a good ambient group for exemplifying the construction of Euler-Arnold's equation or  the connection compatible with a given right-invariant weak Riemannian metric, but not for aspects regarding the geodesic spray. The occurrence of some particular phenomena (existence of an universal cover, for example)  makes this group inappropriate if one wants to grasp the overall picture.  That's why we have to switch to the more general case $G=\DiffM,$ with $M$ a compact manifold.
				   
          \end{rem}

\subsection{The EPDiff equation}	Let us consider  a compact manifold $M$ and	the Lie group $G=\DiffM$ of diffeomorphisms isotopic to the identity. We define an inner product on its Lie algebra $\Gamma(TM)$, the set of all smooth sections of $TM$, via
\begin{equation*}
\langle u_{1},u_{2}\rangle := \int_{M} g\left( Au_{1},u_{2} \right) d\mu \,,
\end{equation*}
where $u_{1}, u_{2} \in \Gamma(TM)$, $g$ is the Riemannian metric on $TM$, $d\mu$, the Riemannian density and the inertia operator
\begin{equation*}
A : \Gamma(TM) \to \Gamma(TM)
\end{equation*}
is a $L^{2}$-symmetric, positive definite, continuous linear operator. In order to get an inner product on each tangent space $T_{\varphi}\DiffM$, we translate the above inner product using the right translations  $R_\varphi$
\begin{equation*}\label{eq:definition-metric}
G_{\varphi}(v_\varphi,w_\varphi)= \int_{M} g\left( A(R_{\varphi^{-1}} v_\varphi), R_{\varphi^{-1}} w_\varphi \right) d\mu \ ,
\end{equation*}
where $v_\varphi,w_\varphi\in T_{\varphi}\DiffM$, and  $A_{\varphi} := R_{\varphi}\circ A\circ R_{\varphi^{-1}}.$

 Let $u(t) := R_{\varphi^{-1}(t)}\dot\varphi(t)$ be the \emph{Eulerian velocity} of the geodesic curve $\varphi(t)$. Then $u(t)$ is a solution of the \emph{Euler-Poincar\'{e} equation \textit{$\mathrm{( EPDiff) }$}  on $\DiffM$}
\begin{equation}\label{eq:EPDiff}
m_{t} + \nabla_{u}m + \left(\nabla u\right)^{t}m + (\mathrm{div} u) m = 0, \quad m := Au \,,
\end{equation}
where $\left(\nabla u\right)^{t}$ is the Riemannian adjoint (with respect to the metric $g$) of $\nabla u$.

When $A$ is invertible, the EPDiff equation~\eqref{eq:EPDiff} can be recast as
\begin{equation}\label{eq:Diff-Euler-Arnold}
u_{t} = - A^{-1}\left\{ \nabla_{u}Au + \left(\nabla u\right)^{t}Au + (\mathrm{div} u) Au \right\},
\end{equation}
which is the \textit{Euler--Arnold} equation for $\DiffM$.

In the sequel we would like to introduce the spray related to the metric $G_\varphi$, but one has to find the best way to visualize it. We need some  terminology from classical mechanics, in order to give an interpretation to some abstract computations. The \textit{configuration space} of the fluid motion (the space of all physically valid states) is $\DiffM$  and every $\varphi\in \DiffM$ takes a  particle $p\in M$ to a particle $\varphi(p)$. Thus every $\varphi$ is a \textit{symmetry} of the mechanical system.
A \textit{motion} of the fluid is a curve $\varphi(t)\in\DiffM$. Obviously during a motion every particle $p$ describes a path $t\to\varphi(t,p)$ and we call $\varphi(t,p)$ the \textit{Eulerian points} of this path. The \textit{Lagrangian velocity field}  $\dot{\varphi}(t,p)$   is the time derivative of this trajectory.  The \textit{Eulerian velocity} is defined as $u(t,q)=\dot{\varphi}(t,\varphi^{-1}(t)(q))$ and is the velocity at time $t$ of the particle  currently in position $q.$   

The covariant derivative along the path $\varphi(t,p)$ is called the \textit{material derivative} and connects  the \textit{Eulerian derivative} $u_t$ with the Lagrangian and Eulerian velocity fields
\begin{equation}\label{materialderivative2}
\left(\nabla_{\raisebox{-2pt}{\tiny\textit t}}\dot{\varphi}(t,p)\right)(t_0)=(u_t+\nabla_uu)(\varphi(t_0,p))\ .\end{equation}

Locally on $TTM$ the second order tangent vector $\ddot \varphi(t,p)$ is expressed as

\begin{equation}\begin{split}
\left(\varphi(t,p),\dot \varphi(t,p), \dot \varphi(t,p), -\Gamma_{\varphi(t,p)}(\dot \varphi(t,p),\dot \varphi(t,p))+ \nabla_{\raisebox{-3pt}{\tiny \textit t}}\dot\varphi(t,p) \right)
\end{split}
\end{equation}
because of $\nabla_{\raisebox{-3pt}{\tiny \textit t}}\dot\varphi(t,p)=\ddot \varphi(t,p)+\Gamma_\varphi(\dot\varphi(t,p),\dot{\varphi}(t,p)).$ Here $\Gamma$ is the Christoffel symbol corresponding to the metric given on $M.$ 

We start now with the connection $\nabla$ of $M$ and we build a connection $\widetilde{\nabla}^q$ on $\DM{q}$, in the way done in Section 9 of \cite{EM1970}. Usually the restriction $q>\frac{3}{2}$ is taken into consideration. The projections $\pi_M, \pi_{TM}$ and  connector  $K$ ascend to smooth mappings $\widetilde{\pi}_{M}^q, \widetilde{\pi}_{TM}^q$ and $\widetilde K^q$, respectively. Here by $\widetilde K^q$ we mean $\widetilde K^q(X):=K\circ X$, when $ X\in TT\DM{q},$ with similar meaning  for $\widetilde \pi_M^q, \widetilde{\pi}_{TM}^q$.

By their very definition:
$$T\mathcal D^q(M)=\{\xi\in H^q(M,TM):\ \widetilde \pi^q_M\circ \xi\in\mathcal D^q(M)\},$$
$$TT\mathcal D^q(M)=\{X\in H^q(M,TTM):\ \widetilde \pi^q_{TM}\circ X\in T\mathcal D^q(M)\},$$
 thus first and second order vector fields on $\DM{q}$ can be seen as mappings on $M$.
Since the connecting morphisms of the inverse systems $\{\DM{q}\}_q$, $\{T\DM{q}\}_q$ and $\{TT\DM{q}\}_q$   are inclusion mappings it is  straightforward to argue the existence of the  inverse limits $\widetilde K=\lim\limits_{\leftarrow q} \widetilde K^q$, $\widetilde \pi _M=\lim\limits_{\leftarrow q} \widetilde \pi^q_M$, and  $\widetilde \pi _{TM}=\lim\limits_{\leftarrow q} \widetilde \pi^q_{TM}$.

We consider now $\varphi(t)\in\DiffM$ a motion of the fluid. The mapping $\widetilde K$ will induce an ILH connection $\widetilde \nabla$ on $\DiffM$, in the sense of \cite{Hideki2,Hideki}. Using $\widetilde \nabla$ we get a covariant derivative along the motion $\varphi(t)$ that satisfies
$$\left(\widetilde\nabla_{\raisebox{-2pt}{\tiny\textit t}}\dot{\varphi}\right)(t_0)(p)=\left(\nabla_{\raisebox{-2pt}{\tiny\textit t}}\dot{\varphi}(t,p)\right)(t_0).$$

Because of \eqref{materialderivative2} one can recast the Euler-Arnold equation \eqref{eq:Diff-Euler-Arnold}   in the form

\begin{equation}\label{materialderivative}
\widetilde\nabla_{\raisebox{-2pt}{\tiny\textit t}}\dot\varphi= S_\varphi(\dot{\varphi})= R_{\varphi} \circ S\circ
R_{\varphi^{-1}} (\dot{\varphi}) ,
\end{equation}
where
\begin{equation*}\label{eq:spray}
S(u) = A^{-1} \left\{ [A,\nabla_{u}] u - (\nabla u )^{t} Au - (\mathrm{div} u) Au \right\}.
\end{equation*}

 Further, we will relate  the above Lagrangian description \eqref{materialderivative} of the EPDiff equation  to the spray equation. First of all, we need an image  of the spray equation. Since the local charts on $\DiffM$ use the Riemannian exponential of $M$, see  \cite{Hideki}, it will  not be helpful to switch from the Eulerian description \eqref{eq:Diff-Euler-Arnold} of the spray equation to a local description on $TT\DiffM$.

The double tangent bundle $TT\DM{q}$ has also a fiber bundle structure over $\DM{q}$ which is identified with $T\DM{q}\oplus T\DM{q}\oplus T\DM{q}$, via Dombrowski's isomorphism \cite{Lang99}. In order to get an appropriate image of the spray equation, let us consider the fiber subbundle of $2$-velocities, see also \cite{Modin}
\begin{equation*}
T^2\DM{q} = \set{\xi\in TT\DM{q} : \ \widetilde \pi^q_{TM}(\xi)=T\widetilde\pi^q_M(\xi)}.
\end{equation*}

The advantage of $T^2\DM{q}$ consists in its vector bundle structure over $\DM{q}$, apart from its canonical identification with $T\DM{q}\oplus T\DM{q}$. Moreover, for $\widetilde \nabla$ exists, one can construct the inverse limit of vector bundles $$T^2\DiffM=\lim\limits_{\leftarrow q}T^2\DM{q}$$ as in \cite{Dodson}, for example. The identification of $T^2\DiffM$ with the Whitney sum $T\DiffM\oplus T\DiffM $ is via
$$(\varphi,\dot\varphi, \ddot\varphi)\to (\dot\varphi, \widetilde\nabla_{\dot\varphi }\dot\varphi).$$

The equation \eqref{materialderivative} corresponds, under the above identification,  to the spray equation.
 Thus, it makes sense to call the mapping
$$v_\varphi \mapsto (v_\varphi,  S_{\varphi} (v_\varphi)),\ \ T\DiffM\mapsto   T\DiffM\oplus T\DiffM $$  the \textit{geodesic spray} related to the right-invariant metric \eqref{eq:definition-metric}. The idea is to prove the existence of geodesics $\varphi^q$ on  Banach approximations $\DM{q}$ and then to obtain geodesics on $\DiffM$ as an inverse limit $\varphi=\lim\limits_{\leftarrow q}\varphi^q$. 
 The big stake is the well-posedness, in the smooth category, of the EPDiff equation on a compact manifold, with an inertia operator $A$ of a pseudo-differential type, an extension of the results presented in \cite{Cismas2016,EKHS}.

In the end, in order to help the reader to follow some of the arguments presented in this paper, we will list in an Appendix some basic facts from  convenient calculus.

\section*{Acknowledgments} I would like to express my very great appreciation to Cornelia Vizman and Boris Kolev for their valuable and constructive suggestions during the planning and development of this research work. The research was partially supported  by CNCS UEFISCDI, project number PN-III-P4-ID-PCE-2016-0778 and partially by Horizon2020-2017-RISE-777911.

\appendix

\section{A glimpse into the convenient calculus}  We add  some details about the convenient calculus of A. Fr\"olicher and A. Kriegl.
All the missing proofs of the statements presented below can be found in \cite{Michor}.

In locally convex spaces there is a weaker notion than that of Cauchy sequences,  namely the  \textit{Mackey-Cauchy sequences}.
\begin{defn} A sequence $(x_n)_n$ in $\E$ is called Mackey-Cauchy if there exists a bounded and absolutely convex set $B$ and for every $\varepsilon> 0$  an integer $n_{\varepsilon}\in\mathbb{N}$ such that
	$$x_n-x_m\in \varepsilon B, \hspace{0,3cm} \forall n>m>n_{\varepsilon}.$$
\end{defn}
This is equivalent with $t_{nm}(x_n-x_m)\rightarrow 0$ for some $t_{nm}\rightarrow  \infty $ in $\mathbb{R}.$ 
\begin{defn} A \textit{convenient vector space} is a locally convex topological vector space which is Mackey complete (every Mackey-Cauchy sequence converges in $\E$).
\end{defn}
Any sequentially complete topological vector space is Mackey-complete.

\begin{defn}(Bornology)  Let $X$ be a set, a \textit{bornology} on $X$ is a collection $\mathcal{B}$ of subsets such that:
	\begin{itemize}
		\item[(i)] $\mathcal{B}$ covers $X,$ i.e.  $X=\bigcup\limits_{B\in \mathcal{B}}B.$
		\item[(ii)] $\mathcal{B}$ is stable under inclusions, if $B\in\mathcal{B}$ and $B_0\subseteq B,$ then $B_0\in \mathcal{B}.$
		\item[(ii)]  $\mathcal{B}$ is stable under finite unions, if $B1,\ldots, B_n\in\mathcal{B},$ then $\bigcup\limits_{i=1}^n B_i\in \mathcal{B}.$
	\end{itemize}
\end{defn}

Given a locally convex space $(\E, \tau)$  we obtain a natural bornology on $\E$ (von Neumann bornology)  consisting of all bounded sets. We call a set $U\subseteq \E$ \textit{bornivorous} if it absorbs every bounded set from its von Neumann bornology.
\begin{defn} A Hausdorff locally convex space $\E$ is called \textit{bornological} if each convex, balanced and bornivorous set in $\E$  is a neighborhood of $0.$
\end{defn}

\begin{defn} Let $ (\E, \tau)$ be a locally convex topological vector space, then the collection of all absolutely convex bornivorous subsets forms a locally convex topology $\tau_{born}$ called the \textit{bornologification} of the initial topology. The space $\E_{born}:=(\E, \tau_{born})$ is called the \textit{attached bornological space} and it is the finest locally convex structure having the same bounded sets as $(\E,\tau).$
\end{defn}

\begin{rem} A subset $U$ in $\E$ is \textit{$c^\infty$-open} iff for every $x\in U$ there is a bornivorous set $B$ such that $x+B\subset U$. Instead, a subset $U$ is open relative to the bornologification of the initial topology iff for every $x\in U$ there exists a convex, balanced and bornivorous set $B$ such that $x+B\subset U.$
	
	A convenient vector space $\E$ interacts naturally with the $c^\infty$-topology: a convenient smooth map between convenient vector spaces will be continuous relative to this topology and usually is not continuous relative to the initial topology on $\E$. But, in general, the $c^\infty$-topology is not a liniar topology\hspace{0,1cm}!  
For linear mappings the above phenomenon is not a problem: a linear mapping between convenient vector spaces is bounded iff it is convenient smooth.
 Besides the cartesian closedness, another cornerstone of the convenient calculus is the fact that the two fundamental spaces $\mathrm{C}^{\infty}(\E,\mathbb{F})$ and $\mathrm{L}(\mathbb{E},\mathbb{F})$ (linear and bounded) will remain in this category, when $\mathbb{E},\mathbb{F}$ are convenient vectors spaces. 
The  smoothness of a curve $c:\RR  \rightarrow\E$ does not depend on the initial topology on $\E,$ it depends only on its bornology.  Usually we substitute the initial locally convex topology with its bornologification and work with bornological locally convex spaces. 
In this way we can exploit the characteristic  property: on bornological spaces a linear mapping is continuous if and only if  it is bounded.
\end{rem}

We equip $\Cinf(\RR,\F)$ with the bornologification of the topology of uniform convergence on compact sets, in all derivatives separately. The space $\Cinf(\E,\F)$ will be equiped with the bornologification of the initial topology relative to all pullback mappings $c^* : \Cinf(\E,\F)\rightarrow \Cinf(\RR,\F),$
$c^*(f):=f \circ c,$ for all $c\in \Cinf(\RR,\E).$ If a locally convex space $\E$ is Mackey-complete (convenient) then  its attached bornological space¸ $\E_{born}$, having the same bounded sets, will be Mackey-complete.

\begin{prop} For locally convex spaces $\E,\F$ we have:
	\begin{itemize}
		\item[(i)] If   $\F$ is a convenient vector space then $\Cinf(\E,\F)$ is a convenient vector space, for any $\E.$ The space $\mathrm{L}(\E,\F)$ is a closed linear subspace and it is a convenient vector space endowed with the initial topology relative to the inclusion mapping. 
		
		\item[(ii)]  If  $\E$ is a convenient vector space then a curve $c:\RR\rightarrow \mathrm{L}(\E,\F)$ is smooth if and only if $t\mapsto c(t)(x)$ is a smooth curve in $\F,$  for all $x\in\E.$ 
	\end{itemize}
\end{prop}

\begin{prop}\label{derivata} Let $\mathbb{E}, \mathbb{F}, \mathbb{G}$ be convenient vector spaces, $U\subset \mathbb{E},$ $V\subset \mathbb{F}$  $c^{\infty}$-open subsets:
	\begin{itemize}
		\item[(i)] If  $f:U\rightarrow \mathbb{F}$ is convenient smooth, then the mapping $df: U\rightarrow\mathrm{L}(  \mathbb{E}, \mathbb{F})$ is convenient smooth and  linear bounded in the second component, where:
		$$d_xf(h):=\frac{d}{dt}\bigg|_{t=0} f(x+th) .$$
		
		\item[(ii)] The differentiation operator $d:\mathrm{C}^{\infty}(U,\mathbb{F})\rightarrow\mathrm{C}^{\infty}(U,\mathrm{L}(\mathbb{E},\mathbb{F}))$ exists, is linear and bounded (smooth) and the chain rule holds:
		$$d_x(f\circ g)(v)=d_{g(x)}f( d_xg(v)).$$
		
		\item[(iii)] Convenient smooth mappings are continuous with respect to the $c^{\infty}$-topology.

		\item[(iv)] Multilinear mappings are convenient smooth if and only if they are bounded and for the derivative we have the product rule:
		$$d_{(x_1,\ldots x_n)}f(v_1,\ldots,v_n)=\sum_{i=1}^nf(x_1,\ldots, x_{i-1},v_i,x_{i+1},\ldots, x_n).$$

		\item[(v)] (Smooth uniform boundedness) A linear mapping $f:\mathbb{E}\rightarrow \mathrm{C}^{\infty}(V,\mathbb{G})$ is convenient smooth (bounded) if and only if $ev_v\circ f : \mathbb{E}\rightarrow\mathbb{G}$ is convenient smooth, for each $v\in V\subset\mathbb{F},$ where $\mathrm{ev}_v:  \mathrm{C}^{\infty}(V,\mathbb{G})\rightarrow\mathbb{G}$ denotes the evaluation mapping.
		
		\item[(vi)] (Smooth detection principle) A mapping $f:U\subset\mathbb{E}\rightarrow \mathrm{L}(  \mathbb{F}, \mathbb{G})$ is convenient smooth if and only if $\mathrm{ev}_y\circ f :U \rightarrow\mathbb{G}$ is convenient smooth for all $y\in\mathbb{F}.$
		
		\item[(vii)] A mapping $ f:U\rightarrow \mathrm{L}(  \mathbb{F}, \mathbb{G})$ is convenient smooth if and only if the mapping $ f:U\rightarrow \mathrm{C}^{\infty}(  \mathbb{F}, \mathbb{G})$ is convenient smooth, i.e. $ \mathrm{L}(  \mathbb{F}, \mathbb{G})\hookrightarrow \mathrm{C}^{\infty}(  \mathbb{F}, \mathbb{G})$ is initial.

	\end{itemize}
\end{prop}

\begin{prop}\label{cartesian closedness}(Cartesian closedness) Let $U_i\subseteq \E_i,$ $i=\overline{1,2},$ be two $\cinf$-open subsets in locally convex spaces which need not to be convenient. 
	Then a mapping $f: U_1\times U_2\rightarrow \F$ is convenient smooth if and only if the cannonically associated mapping $f^{\vee} : U_1\rightarrow \mathrm{C}^{\infty}(U_2, \F)$ exists and is convenient smooth:
	$$\mathrm{C}^{\infty}(U_1\times U_2, \F) =\mathrm{C}^{\infty}(U_1, \mathrm{C}^{\infty}(U_2, \F)).$$
\end{prop}

As a consequence of the cartesian closedness property let us note that the evaluation mapping
$$\mathrm{ev}: \mathrm{C}^{\infty}(U,\mathbb{F})\times U\rightarrow \mathbb{F},  \hspace{0,2cm}\mathrm{ev}(f,x):=f(x),$$
and the composition mapping
$$\circ :  \mathrm{C}^{\infty}(\mathbb{F},\mathbb{G}) \times  \mathrm{C}^{\infty}(U,\mathbb{F})\rightarrow  \mathrm{C}^{\infty}(U,\mathbb{G}) $$
are convenient smooth.

\begin{prop}\label{product rule} Let $f:\mathbb{E}\rightarrow\mathbb{F}$ and $ A:\mathbb{E}\rightarrow \mathrm{L}(\mathbb{F},\mathbb{G})$ be  convenient smooth mappings, then 
	$$d_x(A(\cdot)f(\cdot))v=d_xA(v)(f(x))+A(x)(d_xf(v)),$$
	for all $x,v\in \mathbb{E}.$
\end{prop}

\begin{proof} The evaluation mapping $\mathrm{ev}:\mathrm{C}^{\infty}(\mathbb{F},\mathbb{G})\times \mathbb{F}\rightarrow \mathbb{G}$ is convenient smooth and the curve $c:\mathbb{R}\rightarrow\mathrm{L}(\mathbb{F},\mathbb{G})$ is smooth iff $c:\mathbb{R}\rightarrow\mathrm{C}^{\infty}(\mathbb{F},\mathbb{G})$ is smooth by Proposition \ref{derivata} (vii). Thus $\mathrm{ev}:\mathrm{L}(\mathbb{F},\mathbb{G})\times \mathbb{F}\rightarrow \mathbb{G}$ is convenient smooth and bilinear. Hence
	$$d_x(A(\cdot)f(\cdot))v=d_x(\mathrm{ev}(A(\cdot),f(\cdot)))(v)=d_{(A(x),f(x))}\mathrm{ev}(d_xA(v), d_xf(v))$$
	$$=\mathrm{ev}_{d_xf(v)}A(x)+\mathrm{ev}_{f(x)}d_xA(v)=d_xA(v)(f(x))+A(x)(d_xf(v)),$$
	using Proposition \ref{derivata} (iii).
\end{proof}

\begin{rem}\label{remarca metrica} The identity is also true for $L^k(\E,\F)$ instead of $L(\E,\F).$
	
\end{rem}

\section{Vector bundles over a convenient manifold}

For $x\in\mathbb E$ the \textit{kinematic tangent vector} with foot point $x$ is the pair $(x,X),$  $X\in \mathbb E.$ The space $T_x\E=\mathbb E$ of kinematic tangent vectors with foot point $x$ consists of all derivatives $c^{'}(0)$
of the smooth curves $c:\RR\rightarrow\E$ with $c(0)=x.$ For a convenient smooth mapping $f:\E\rightarrow\F$ the \textit{kinematic tangent mapping} at $x$ is defined by
$$T_xf: T_x\mathbb E\rightarrow T_{f(x)}\F,\hspace{0,2cm}T_xf(x,X):=(f(x),d_xf(X)).$$
If $M$ is a convenient smooth manifold on the set
$$\bigcup_{\alpha\in A}\mathcal U_{\alpha}\times\E_{\alpha}\times \{\alpha\},$$
we consider the equivalence relation
$$(p,v,\alpha)\sim (q,w,\beta) \iff p=q, \hspace{0,2cm}\mathrm{and}\hspace{0,2cm}d_{\f_{\raisebox{-2pt}{\tiny $\beta$}}(p)}(\f_{\alpha\beta})w=v$$
and denote the quotient set by $TM,$ the \textit{kinematic tangent bundle} of $M.$
We define $\pi_M:TM\rightarrow M$ by $\pi_M([p,v,\alpha])=p$ and $T\mathcal U_{\alpha}:=\pi_M^{-1}(\mathcal U_{\alpha})\subset TM.$ The mapping $Tu_{\alpha}:T\mathcal U_{\alpha}\rightarrow u_{\alpha}(\mathcal U_{\alpha})\times \E_{\alpha}$ defined by
$$Tu_{\alpha}([p,w,\beta])=(u_{\alpha}(p), d_{u_{\raisebox{-2pt}{\tiny $\beta$}}(p)}(u_{\alpha\beta})w )$$
is giving a chart for an atlas $(T\mathcal U_{\alpha}, Tu_{\alpha})_{\alpha\in A}$ of $TM.$

The set $T_pM:=\pi^{-1}_M(p)$ is called the \textit{fiber} over $p$ of the tangent bundle. It carries a canonical convenient vector space structure induced by
$$T_pu_{\alpha}:=Tu_{\alpha}|_{T_pM}: T_pM\rightarrow\{p\}\times\E_{\alpha}\cong \E_{\alpha},$$
for $p\in \mathcal U_{\alpha}.$ For connected convenient manifolds, e.g. $\DiffS$,  the fiber of the tangent bundle coincides with the modelling space. The same observation holds, in particular, for the Lie algebra of a connected Lie group.

The kinematic tangent bundle can be also defined as the quotient of the space $C^ {\infty}(\RR, M)$ by the equivalence relation:  $c_1\sim c_2 \iff c_1(0)=c_2(0)$ and in each chart $(\mathcal U_{\alpha},u_{\alpha})$ with $c_1(0)=c_2(0)\in \mathcal U_{\alpha}$ we have $ (u_{\alpha}\circ c_1)^{'}(0)= (u_{\alpha}\circ c_2)^{'}(0).$ In this way any curve $c\in \mathrm{C}^ {\infty}(\RR, M)$ corresponds to the kinematic tangent vector $[c(0),(u_{\alpha}\circ c)^{'}(0),\alpha ].$ For a convenient smooth mapping $f:M\rightarrow N$ the tangent mapping $Tf$ will send the equivalence class $[c]$ in the equivalence class $[f\circ c]$ and its local representative with respect to some charts is the kinematic tangent mapping of the local representative of $f.$

\begin{rem}On convenient vector spaces another kind of tangent vectors are available: the operational tangent vectors, see Section 28.1 in \cite{Michor}. The two notions will not coincide in general and will give two different tangent bundles of a convenient manifold. This difference causes some headaches and leads to the existence of $12$ different notions of differential forms in the convenient setting. The "right" notion for a convenient manifold is the kinematic tangent bundle, the other one does not even preserve products or there exist  no vertical lifts.  Anyway, for manifolds modelled on nuclear Fr\'{e}chet spaces the two notions coincide and one recovers the result from the finite dimensional case: any tangent vector is a derivation.
\end{rem}

In the convenient setting it's worth reducing everything to curves, thus the convenient vector space  $C^\infty(\RR,\,\mathbb E)$ will play a central role. First of all, the space of $\E$-valued functions on a manifold $M$ will be endowed with the initial topology with respect to the mappings
$$C^\infty(M,\mathbb E)\xrightarrow{(i_{\mathcal U_\alpha})^*}C^\infty (\mathcal U_\alpha, \mathbb E)\xrightarrow{(u_\alpha^{-1})^*}C^\infty(U_\alpha,\mathbb E)\xrightarrow{c^*} C^\infty(\RR,\mathbb E) $$

The space of smooth sections of $TM$  is endowed with the topology given by the closed embedding
$$\Gamma(TM)\to \prod_\alpha C^\infty (\mathcal U_\alpha , \mathbb E)$$
$$s\to (pr_2\circ \psi_\alpha \circ s|_{\mathcal U_\alpha})_\alpha$$
which is actually the initial topology with respect to the family of mappings

$$\Gamma(TM)\xrightarrow{\left(i_{\mathcal  U_\alpha}\right)^* }\Gamma(TM|_{\mathcal U_\alpha})\xrightarrow{(pr_{\raisebox{-1pt}{\tiny 2}}\circ \psi_\alpha)_*} C^\infty (\mathcal U_\alpha, \mathbb E)\xrightarrow{(u_\alpha^{-1})^*}C^\infty(U_\alpha,\mathbb E)\xrightarrow{c^*} C^\infty(\RR,\mathbb E)   $$
where $i_{\mathcal U_\alpha}$ is the restriction to $\mathcal U_\alpha$ and $\psi_\alpha$ is a local trivialization of $TM.$

 Let now $\pi: E\rightarrow M$ be a convenient smooth mapping between convenient smooth manifolds. By a \textit{vector bundle chart} on $(E,\pi, M, \F)$ we mean a pair $(\mathcal U,\p),$ where $\mathcal U$ is an open subset in $M,$ and where $\p$ is a fiber respecting diffeomorphism $\psi:E_{\mathcal U}:=\pi^{-1}(\mathcal U)\rightarrow \mathcal U\times \E$ such that
$$pr_1\circ \psi =\pi, $$ where $\E$ is a fixed convenient vector space, called the \textit{standard fiber}.

Two vector bundle charts $(\mathcal U_{\alpha},\p_{\alpha}), (\mathcal U_{\beta},\p_{\beta})$ are called \textit{compatible}, if the mapping $\p_1\circ\p_2^{-1}$ is a fiber linear isomorphism
$$\p_{\alpha}\circ\p_{\beta}^{-1}(p,v)=(p, \p_{\alpha\beta}(p)v), \quad\quad v\in \E,$$
for some mapping $\p_{\alpha\beta}: \mathcal U_{\alpha\beta}:= \mathcal U_{\alpha}\cap \mathcal U_{\beta}\rightarrow GL(\E).$ The mapping is then unique and convenient smooth into $L(\E,\E),$ and is called the \textit{transition function} between the two vector bundle charts.

\begin{rem}\label{remarca Lang99}Compare this definition with the definition of a  Banach vector bundle, presented in \cite{Lang99}. An extension of a result known for  Banach spaces (\cite{Lang99} Proposition III.1.1, \cite{Hideki} Theorem 5.3) holds
	if the mapping $f: U\times \E\rightarrow \F$ is a convenient smooth  and  linear in the second argument, then the mapping of $U$ into $L(\E,\F),$ $x\mapsto f(x,\cdot),$ is a convenient smooth mapping. The converse  also holds, by Proposition \ref{cartesian closedness}, since $L(\E,\F)\subset \mathrm{C}^{\infty}(\E,\F) $ is initial. In \cite{Lang99} the author has omitted the veracity of the result for infinite dimensional Banach spaces  and thus the conditions VB1 and VB2, in Chapter III,  are enough to define a Banach vector bundle. With VB1 and VB2 in Definition III.1 of \cite{Lang99} we obtain the above formulation for  Banach manifolds.   

\end{rem}

If $( E, \pi, M, \E)$ is a convenient smooth vector bundle with a vector bundle atlas $(\f_{\alpha}, \pi^{-1}(\mathcal U_{\alpha}))_{\alpha\in A},$ then we define the dual vector bundle
$$E^{'}:=\bigcup_{p\in M}E_p^{'},$$ with the standard fiber the bornological dual $\E^{'}$ and the transition functions
$$\psi_{\alpha\beta}(p):=(\f_{\beta\alpha}(p))^t,$$ 
naturally obtained using the transpose mapping relative to the bornological duals.
For two convenient smooth vector bundles  $( E, \pi_1, M, \E)$ and $( F, \pi_2, M, \F )$  with $( \pi_1^{-1}(\mathcal U_{\alpha}), \f_{\alpha})_{\alpha\in A_1},$ and $( \pi_2^{-1}(\mathcal V_{\alpha}), \phi_{\alpha})_{\alpha\in A_2}$ the corresponding vector bundle atlases, one can construct another vector bundle over $M,$ the Hom-bundle
$$L(E,F):=\bigcup_{p\in M}L(E_p, F_p),$$ having the standard fiber the convenient vector space  $L(\E,\F).$ 
The transition functions are
$$\psi_{\alpha\beta}(p)(T):= \phi_{ \alpha\beta}(p)\circ T\circ \f_{\alpha\beta}^{-1}(p),\hspace{0,5cm}T\in L(\E, \F).$$ 

With this terminology we have $E^{'}:=L(E,M\times\RR).$ We are ready to  define now the \textit{kinematic cotangent bundle} $T^{'}M,$ having the transition functions
$$\psi_{\alpha\beta}(p):=T_{\f_{\alpha}(p)}(\f_{\beta}\circ \f_{\alpha}^{-1})^t\in GL(\E^{'})\subset L(\E^{'},\E^{'}).$$

If we use the G\^ateaux smoothness  to define a manifold modelled by locally convex spaces then we can not define differential forms as G\^ateaux smooth sections of a vector bundle, see \cite{Neeb} for a discussion. This is the case because for non normable locally convex spaces the evaluation mapping: $$\mathrm{ev}:\E\times \E^{'}\rightarrow\RR,$$ is not continuous for any linear topology on $\E^{'},$ by a theorem of B. Maissen  \cite{Maissen}.

In the convenient setting a kinematic differential $k$-form is defined as a convenient smooth section of the vector bundle $L^k_{alt}(TM, M\times\RR):$
$$\Omega^k(M):=\Gamma(L^k_{alt}(TM, M\times\RR)),$$
with the modelling space $L^k_{alt}(\E,\RR),$ the space of bounded $k$-linear alternating mappings, where $\E$ is the modelling space of $M.$ This construction is the only one which is invariant under Lie derivatives, pullbacks or exterior derivatives. There are a lot of other candidates but all have major drawbacks, see Section 33 of \cite{Michor} for a discussion.

\begin{rem}\label{naturalpairing}The reason why this construction is possible is the following: the evaluation mapping $\mathrm{ev}:\E\times \E^{'}\rightarrow\RR$ is always convenient smooth, thus continuous relative to the $\cinf$-topology on $\cinf$ $ (\E\times\E^{'}).$ But $\cinf$ $(\E\times\E^{'})$ is not a topological vector spaces in general (if $\E$ is not normable), and thus $\mathrm{ev}$ is not continuous relative to some linear topology on the space $\E^{'},$ to avoid any contradiction with  Maissen's theorem. 
\end{rem}

The space $\Omega^k(M)$ carries the structure of a convenient vector space, induced by the closed embedding
$$\Omega^k(M)\to \prod_{\alpha} C^\infty (\mathcal U_\alpha, L^k_{alt}(\E,\RR)),$$
$$\omega \mapsto  pr_2\circ \psi_\alpha\circ (\omega|_{\mathcal U_\alpha}), $$
 having the initial topology induced by the mappings
$$\hspace{-0,25cm}\Gamma(L^k_{alt}(TM, M\times\RR))\xrightarrow{(pr_{\raisebox{-1pt}{\tiny 2}}\circ \psi_\alpha)_*\circ \left(i_{\mathcal  U_\alpha}\right)^* }C^\infty (\mathcal U_\alpha, L^k_{alt}(\E,\RR))\xrightarrow{c^*\circ(u_\alpha^{-1})^*}C^\infty(\RR,L^k_{alt}(\E,\RR))   $$
where $u_\alpha:\mathcal U_\alpha\to \E$ is a smooth atlas of $M$, $\psi_\alpha$ the induced vector bundle chart, a similar construction to the one for $\Gamma(TM).$
In the same manner the space $\Omega^k(M,V)$,  of differential forms with values in a convenient vector space $V$, becomes a convenient vector space.

\begin{rem}
 The following mappings are convenient smooth:
 $$\ \ d:\Omega^k(M)\to \Omega^{k+1}(M) \ \ \  \text{(exterior differentiation operator)}$$
 $$i:\Gamma(TM)\times\Omega^k(M)\to \Omega^{k-1}(M)\ \ \ \ \ \ \text{(insertion operator)}$$
 $$f^*:\Omega^k(M)\to \Omega^{k}(N)\hspace{2,5cm} \text{(pullback operator)}$$

\end{rem}

	\medskip
	Received xxxx 20xx; revised xxxx 20xx.
	\medskip

\end{document}